\documentclass{amsart}

\usepackage{color}
\usepackage{amssymb, amsthm, amsmath, amscd}

\usepackage[pdftex]{graphicx}
\usepackage{float}
\usepackage[below]{placeins}

\newtheorem{lemma}{Lemma}
\newtheorem{proposition}[lemma]{Proposition}
\newtheorem{theorem}[lemma]{Theorem}

\newtheorem{definition}[lemma]{Definition}

\newcommand{\wilde}{\widetilde}

\newcommand{\btau}{\bar{\tau}}
\newcommand{\R}{\mathbf R}

\newcommand{\lin}{\mathcal{L}}
\newcommand{\mcal}{\mathcal{M}}

\newcommand{\eps}{\varepsilon}
\newcommand{\ubar} {\underline}
\newcommand{\bfa}{\vec{a}}
\newcommand{\bfb}{\mathbf{b}}

\begin{document}

\title[Doubly periodic solitons for the MCF]{Doubly periodic self-translating surfaces for the mean curvature flow}
\author{Xuan Hien Nguyen}
\address{Department of Mathematics, Iowa State University, Ames, IA 50011} 
\email{xhnguyen@iastate.edu}

\maketitle

\begin{abstract}
We construct new examples of self-translating surfaces for the mean curvature flow from a periodic configuration with finitely many grim reaper cylinders in each period. Because this work is an extension of the author's article on the desingularization of a finite family of grim reaper cylinders, we simply discuss the ideas of the construction and here prove only that the periodic configuration has the necessary flexibility. These examples show that self-translating surfaces do not necessarily have quadratic volume growth rate in contrast to self-shrinking surfaces.  
\end{abstract}

\section{introduction}
A surface $M \subset \R^3$ is said to be a \emph{self-translating solution} or \emph{soliton} if it satisfies 
	\begin{equation}
	\label{eq:soliton}
	H - \vec e_{y} \cdot \nu = 0,
	\end{equation}
where $H$ is the mean curvature, $\nu$ is the unit normal vector such that $\mathbf H = H \nu$ and $\mathbf H$ is the mean curvature vector. Without loss of generality, we have chosen the velocity of the translation to be $\vec e_y$, the third coordinate vector. As the name indicates, a surface $M$ is self-translating if and only if the $t$-time slice of a mean curvature flow starting at $M$ is $M_t = M + t \vec e_y$.

The cylinder over the grim reaper given by 
		\begin{equation}
		\label{eq:Gamma}
		\wilde \Gamma = \{ (x, y, z) \in (-\pi/2, \pi/2) \times \R^2 \mid y = -\ln(\cos x) \}
		\end{equation}
is an example of self-translating surface in $\R^3$. Let us denote by $\wilde \Gamma_n$ the grim reaper cylinder $\wilde \Gamma$ shifted by $(\tilde b_n, \tilde c_n, 0)$, for real numbers $\tilde b_n$ and $\tilde c_n$. With a slight abuse of language, we will drop the term ``cylinder" and also call $\wilde \Gamma$ and its translated apparitions grim reapers. Other examples of solitons include rotationally symmetric surfaces \cite{altschuler-wu;translating-surfaces}, a non-graphical genus zero rotationally symmetric soliton \cite{clutterbuck-schnurer-schulze}, and examples desingularizing the intersection of a grim reaper and a plane \cite{mine;tridents} or a finite family of grim reapers in general position \cite{mine;self-trans}. In higher dimensions, the existence of convex non $k$-rotationally symmetric self-translating surfaces in $\R^n, n \geq k$, are given by Wang \cite{wang;convex-solutions}.  

In this paper, we show the existence of a family of doubly periodic self-translating surfaces. The period in the $\vec e_x$ direction can be taken to be any real number, except $ \pm \frac{\pi}{q}$, $q \in \mathbf N$. 
\begin{theorem}
\label{thm:cor}
Given a vector $\bfa = (a_x, a_y, 0) \in \R^3$ with $a_x \neq 0$ and $a_x \neq  \pm \frac{\pi}{q}$, $q \in \mathbf N$, there exists a one parameter family of surfaces $\{\wilde M\}_{\btau \in (0, \delta_{\btau})}$, with $\delta_{\bar \tau}>0$ depending only on $\bfa$ such that $\wilde M_{\btau}$ is a  surface satisfying \eqref{eq:soliton} that is invariant under the translation by $\bfa$ and periodic of period $2 \pi \bar\tau$ in the $z$-direction. 
\end{theorem}
	
The surfaces from Theorem \ref{thm:cor} show that in $\R^3$, self-translating surfaces can have cubic area growth. We contrast this with a special case of the result from \cite{ding-xin;volume-growth}, which gives quadratic area growth for any complete noncompact self-shrinking surface in $\R^3$.

Theorem \ref{thm:cor} is a corollary our main result, Theorem \ref{thm:main}. As an extension of the construction from \cite{mine;self-trans}, we show below that given a periodic configuration $G=\bigcup_{j \in \mathbf Z} \bigcup_{n=1}^{N_{\Gamma}}(\wilde \Gamma_n + j \bfa)$ in general position, it is possible to desingularize $G$ to obtain self-translating surfaces. One proves Theorem \ref{thm:cor} by choosing an adequate family $\bigcup_{n=1}^{N_{\Gamma}}\wilde \Gamma_n $ depending on $\bfa$. Note that if $a_x = \pm \frac{\pi}{q}$, we would immediately have coinciding asymptotic planes after $q$ periods and $G$ would not be in general position as defined below. Let us assume from now on that $\bfa \cdot \vec e_z=0$ and $N_{\Gamma}$ is finite.

\begin{definition}
\label{def:gen-pos}
A periodic family of grim reapers $G=\bigcup_{j \in \mathbf Z} \bigcup_{n=1}^{N_{\Gamma}}(\wilde \Gamma_n + j \bfa)$ is said to be in \emph{general position} if it satisfies the following requirements:
	\begin{enumerate}
	\item there exists a constant $\eps>0$ such that the distance between any two asymptotic planes to $G$ is at least $\eps$,
	\item no three grim reapers intersect on the same line.
	\end{enumerate}
\end{definition}

Let us denote by $\delta$ the minimum of the distance between two intersection points of $G \cap \{z=0\}$ and by $\delta_{\Gamma}$ the minimum measure of the angles formed by intersecting grim reapers and of the angles the grim reapers form with $\vec e_y$ at the intersection points. Note that both $\delta$ and $\delta_{\Gamma}$ are positive because there are finitely many intersections within a period.

We denote by $N_{I}$ the number of intersection lines within a period of $G$ (i.e. between the planes $x=c$ and $x=c+\bfa \cdot \vec e_x$, where $c$ is taken generically). The number $N_I$ is different from the number of intersections of $\bigcup_{n=1}^{N_{\Gamma}}(\wilde \Gamma_n + j \bfa)$. We associate a positive integer $m_k$ to each intersection line, which allows us to take different scales at different intersections, making the construction as general as possible.

\begin{theorem}
\label{thm:main}
Given a periodic family of grim reapers  $G= \bigcup_{j \in \mathbf Z} \bigcup_{n=1}^{N_{\Gamma}}(\wilde \Gamma_n + j \bfa)$ in general position, there is a one parameter family of surfaces $\{ \wilde{\mcal}_{\bar \tau} \}_{\bar \tau \in (0, \delta_{\bar \tau})}$, with $\delta_{\bar \tau}$ depending only on $N_I$, $\max_k(m_k)$, $\delta$, and $\delta_{\Gamma}$ satisfying the following properties:
 	\begin{itemize}
	\item $\wilde \mcal_{\bar \tau}$ is a complete embedded surface satisfying \eqref{eq:soliton}.
	\item $\wilde \mcal_{\bar \tau}$ is singly periodic of period $2 \pi \bar\tau$ in the $z$-direction.
	\item $\wilde \mcal_{\bar \tau}$ is invariant under the translation $(x,y,z) \mapsto (x + \bfa \cdot \vec e_x, y+ \bfa \cdot \vec e_y, z)$.
	\item If $U$ is a neighborhood in $\R^2$ such that $U \times \R$ contains no intersection line, then $\wilde \mcal_{\bar \tau} \cap (U \times \R)$ converges uniformly in $C^r$ norm, for any $r<\infty$, to $G \cap ( U \times \R)$ as $\btau \to 0$. 
	\item If $T_{kj}$ is the translation parallel to the $xy$-plane that moves the $k$th intersection line of the $j$th period to the $z$-axis, then $\btau^{-1} T_{kj} (\wilde \mcal_{\btau})$ converges uniformly in $C^r$, for any $r<\infty$, on any compact set of $\R^3$ to a Scherk surface with $m_k$ periods between $z=0$ and $z=2 \pi$ as $\btau \to 0$. 
	\end{itemize}
\end{theorem}

The proof of Theorem \ref{thm:main} is similar to the one from \cite{mine;self-trans}, where we desingularized a finite family of grim reapers in general position. Many of the steps from the previous article (construction of desingularizing surfaces and initial approximate solutions, study of the associated linear operator, and application of fixed point theorem) are valid here provided one has enough flexibility to change the tangent vectors at each intersection slightly while keeping the initial configuration embedded. The flexibility is needed to deal with the small eigenvalues of the linear operator associated to $H-\vec e_y \cdot \nu$.  In \cite{mine;self-trans}, we started by fixing the left sides of the grim reapers and propagating the errors caused by the flexibility to the right. The family of grim reapers was finite, so the accumulation of errors was bounded. In the present proof, we move through the configuration from the bottom up. We start by fixing well chosen grim reaper edges at the bottom and let the free grim reaper ends absorb the movement from the dislocations at the intersections. 

This article is to be read as an extension of \cite{mine;self-trans} and we rely on the reader's familiarity  with either \cite{mine;self-trans} or \cite{kapouleas;embedded-minimal-surfaces}. The rest of the article is structured as follows. In the next section, we prove Theorem \ref{thm:cor} assuming Theorem \ref{thm:main}. In Section \ref{sec:init-conf}, we show that our periodic configuration has enough flexibility. In Section \ref{sec:perturbation}, we outline the main steps of the proof of Theorem \ref{thm:main}. 

\subsection*{Acknowledgment} The present paper is the result of a conversation with Sigurd Angenent on whether the mean curvature $H$ on $M \setminus B_R$ tends to $0$ as $R \to \infty$ for  a self-translating surface $M$. The author would like to thank him for many discussions and suggesting this counterexample. 


\section{Proof of Theorem \ref{thm:cor} given Theorem \ref{thm:main}}
\label{sec:proof-cor}

We show that for a vector $\bfa$ as in Theorem \ref{thm:cor}, we can find a periodic family of grim reapers in general position invariant under translation by $\bfa$. To simplify the proof, we assume without loss of generality that $a_x>0$.

\subsection*{Case 1: $\pi>a_x  > \pi/2$} We just take $G = \bigcup_{j \in \mathbf Z} \wilde \Gamma + j \bfa$. The $j$th grim reaper $\wilde \Gamma + j \bfa$ intersects only its neighbors $\wilde \Gamma + (j\pm 1) \bfa$ and the distance between two asymptotic planes is bounded below by $\min(\pi - a_x, 2a_x-\pi)$.

\subsection*{Case 2: $a_x> \pi$} Let $K$ be a positive integer such that $K \pi > a_x > K \pi/2$ (for example, $K := \left\lfloor \frac{a_x}{\pi} \right\rfloor +1$). We can fall back on the first case by considering $\bigcup_{j \in \mathbf Z} \left(\bigcup_{i=0}^{K-1} \wilde \Gamma + \tfrac{i + Kj}{K} \bfa \right)$ but the configuration is really invariant under translation by $\bfa/K$. To kill all periods smaller than $\bfa$, we take $\{\mathbf d_i\}_{i=0}^{K-1}$ a set of distinct vectors in $\R^3$ for which $\mathbf d_0=\mathbf 0$, $\mathbf d_i \cdot \vec e_z=0$, and $|\mathbf d_i| < \frac{1}{10}\min(\pi-\frac{a_x}{K}, \frac{2a_x}{K} - \pi) $ for $i=0,\ldots, K-1$. The configuration
	\[
	G = \bigcup_{j \in \mathbf Z} \left( \bigcup_{i=0}^{K-1} \wilde \Gamma + i \frac{\bfa}{K} + \mathbf d_i+ j \bfa \right)
	\]
yields a self-translating surface invariant under the translation by $\bfa$ if $\btau$ is small enough. Indeed, Theorem \ref{thm:main} states that the smaller $\btau$ is, the closer the constructed self-translating surface is to the initial configuration $G$. Therefore, for $\btau$ small enough, the self-translating surface will not be invariant under translations by $i \frac{\bfa}{K}$, $i=1, \ldots, K-1$.

\subsection*{Case 3: $\pi/2 > a_x>0$ with $a_x \neq \pm \frac{\pi}{q}, q \in \mathbf N$} In this case, the periodic family of grim reapers is just
	\[
	G= \bigcup_{j \in \mathbf Z} \wilde \Gamma + j \mathbf a.
	\]
We now show that it is in general position. Let $K$ be the smallest integer so that $K a_x> \pi$. From the conditions on $a_x$, $(K-1) a_x \neq \pi$ and the distance between any two asymptotic planes of $G$ is either $a_x, \pi-(K-1) a_x$, or $K a_x - \pi$. The last step is to prove that there are no triple intersection by contradiction. In what follows, we consider $G \cap \{ z=0\}$ to simplify the vocabulary.  Without loss of generality thanks to the periodicity, we can assume that the left-most grim reaper is $\wilde \Gamma$ and that for two integers $i,j$ with $0<i<j<K$, we have  $\wilde \Gamma \cap (\wilde \Gamma + i \bfa) \cap (\wilde \Gamma + j \bfa)=\{p\}$. Using the periodicity of $G$, we have that $(\wilde \Gamma + (j-i) \bfa) \cap (\wilde \Gamma+ j \bfa) = \{p + (j-i) \bfa\}$ and $(\wilde \Gamma + j \bfa) \cap (\wilde \Gamma + (j+i) \bfa) = \{p+ j \bfa\}$. The grim reaper $\wilde \Gamma + j \bfa $ passes through the three collinear points $p$, $p+(j-i)\bfa$, and $p+ j \bfa$, which is not possible. Therefore there can be no triple intersection.


\section{Flexibility of the initial configuration}
\label{sec:init-conf}

In this section, the third dimension does not add any information, so we work with cross-sections in the $xy$-plane. We start by giving uniform lower bounds on all the intersection angles. 
\begin{lemma}
\label{lem:2}
Given $G = \bigcup_{j \in \mathbf Z} \bigcup_{n=1}^{N_{\Gamma}}(\wilde \Gamma_n + j \bfa)$ a family of grim reapers in general position,  there exist positive numbers $\delta$ and $\delta_{\Gamma}$ such that
	\begin{itemize}
	\item  the four angles formed at the intersection of any two grim reapers are greater than $30\delta_{\Gamma}$,
	\item any tangent vector to a grim reaper at an intersection forms an angle greater than $30 \delta_{\Gamma}$ with $\vec e_y$,	
	\item the arc length distance on the grim reapers between any two intersection points of $G \cap \{z=0\}$ is greater than $2\delta$. 
	\end{itemize}
\end{lemma}

\begin{proof}
Let us recall that each $\wilde \Gamma_n$ is the surface $\wilde \Gamma$ from \eqref{eq:Gamma} shifted by $(\tilde b_n, \tilde c_n, 0)$. Let $K$ be the smallest integer for which 
	\[
	K |a_x| \geq \left(\pi+ \max_{1 \leq n \leq N_{\Gamma}} \tilde b_n - \min_{1 \leq n \leq N_{\Gamma}} \tilde b_n\right).
	\]
In Lemma 3 in \cite{mine;self-trans}, we proved these properties for a finite family of grim reapers in general position. Lemma \ref{lem:2} here follows because $\bigcup_{1 \leq j \leq K} \bigcup_{n=1}^{N_{\Gamma}}(\wilde \Gamma_n + j \bfa)$ is a finite family in general position and contains all the relevant intersection points. Note that part (i) of Definition 2 in \cite{mine;self-trans} requires that $|\tilde b_n - \tilde b_m - k \pi| > \eps$ for $n\neq m$ and for all integers $k$. This is more restrictive than necessary and can be replaced by (i) from Definition \ref{def:gen-pos} without any changes in the rest of \cite{mine;self-trans}. Moreover, the angle $\pi/2 - 30 \delta_{\Gamma}$ should be replaced by $\pi/2$ in (i) Lemma 3 of \cite{mine;self-trans}. 
\end{proof}

\subsection{Construction of a flexible initial configuration}

There are exactly four unit tangent vectors emanating from each intersection point $p_k$, $k = 1, \ldots, N_I$. We  denote them by $v_{k1}$, $v_{k2}$, $v_{k3}$, and $v_{k4}$, where the number refers to the order in which they appear as we rotate from $\vec e_y$ counterclockwise. The goal of this section is to perturb $G$ into a configuration $\overline G$ for which the tetrad of  directing vectors $\bar T_k=\{\bar v_{k1}, \bar v_{k2}, \bar v_{k3}, \bar v_{k4}\}$ satisfies
	\[
	\theta_1(\bar T_k) : = \angle(-\bar v_{k1}, \bar v_{k3}) = 2 \theta_{k,1}, \quad \theta_2 (\bar T_k) : = \angle (-\bar v_{k2}, \bar v_{k4}) = 2 \theta_{k,2},
	\]
where $\theta_{k,1}$ and $\theta_{k,2}$ are two small angles which measure how much the pairs of vectors $(\bar v_{k1},\bar  v_{k3})$ and $(\bar v_{k2}, \bar v_{k4})$  fail to point in opposite directions.
The process of changing the directing vectors at an intersection is called \emph{unbalancing} and will help us deal with small eigenvalues of the linear operator $L=\Delta+|A|^2$, which is the operator associated to normal perturbations of the mean curvature $H$. As we scale down the desingularizing Scherk surfaces to fit into a small neighborhood of the intersection lines, the contribution of $H$ dominates and the equation $Lv=E$ will capture the behavior of the linearized equation associated to \eqref{eq:soliton} on the desingularizing surface. The operator $L$ has a kernel of dimension $3$ generated by $\vec e_x \cdot \nu$, $\vec e_y \cdot \nu$, and $\vec e_z \cdot \nu$. We can rule out $\vec e_z \cdot \nu$ by imposing a symmetry, namely that the final surface be invariant under the reflection $z \mapsto -z$. To deal with the other two eigenfunctions at each intersection, one can add or subtract a linear combination of eigenfunctions to the inhomogeneous term E in order to land in the
space perpendicular to the kernel, where $L$ is invertible. The key to a successful
construction is to be able to generate these linear combinations by slight unbalancing of the
 initial configuration.

\begin{proposition}
Given $\ubar \theta:=\{ \theta_{k,1}, \theta_{k,2}\}_{k=1}^{N_I}$, one can perturb $G$ into a configuration $\overline G$ so that the tetrad of directing vectors $\bar T_k=(\bar v_{k1}, \bar v_{k2}, \bar v_{k3}, \bar v_{k4})$ at $\bar p_k$, the $k$th intersection point of $\overline G$, satisfies $\theta_{1}(\bar T_k) = \theta_{k,1}$ and $\theta_2(\bar T_k) = \theta_{k,2}$ for $k = 1, \ldots,N_I$. Moreover, there exists a constant $\delta'_{\theta}>0$ depending only on $N_I$, $\delta_{\Gamma}$, and $\eps$ such that $\overline G$ is embedded if $|\ubar \theta|<\delta'_{\theta}$.
\end{proposition}

 Before starting the proof, we give some terminology. We say that two points $p$ and $q$ are equivalent and write $p \sim q$ if there is a $k \in \mathbf Z$ for which  $p + k \bfa =q$. In the proof below, we take the quotient of the configurations with respect to this equivalence and denote both the configuration and the corresponding quotient by the same letter. The \emph{edges} of a configuration $G$ are the closures of the bounded connected pieces of $G \setminus \{\textrm{intersection points of }G\}$. The \emph{rays} are the closure of the unbounded pieces. We say two edges (rays, or grim reapers) are equivalent when one is the image of the other under translation by a multiple of $\bfa$.

In the case of finitely many grim reapers, we started by unbalancing at intersection points on the left side and propagated the perturbations to the right. Because there were finitely many intersection points, we accumulated finitely many error terms, which could then be bounded. Here, the strategy is to fix a low edge on each grim reaper then modify the intersection points and pieces of grim reapers as we move upward.

\begin{proof} In this proof, we denote with an overline sets of edges that will not be changed further. We will build the final configuration $\overline G$ in pieces, starting with $\overline{\mathcal B}_1$, then $\overline G_1$, $\overline{\mathcal B}_2$, $\overline G_2$ ... in finite number of steps.  	

{\bf Part 1.} We start by defining the bottom of the configuration 
	\begin{equation}
	\label{eq:bottom1}
	\overline{\mathcal B}_1= \{ q \in G \mid y(q) = \min_{q' \in G, x(q')= x(q)} y(q')\}.
	\end{equation}
Because $G$ is periodic and connected, $\overline{\mathcal B}_1$ is connected and composed of edges from distinct grim reapers. We can assume by renumbering if necessary, that $\overline{\mathcal B}_1$ contains edges from the first $N_{1}$ grim reapers and also contains the first $N_{1}$ intersection points, $1 \leq N_{1} \leq N_{\Gamma}$. All the intersection points $\bar p_k:=p_k$, $k =1, \ldots, N_{1}$ are now fixed as well as the second and third directing vectors, $\bar v_{k2}:=v_{k2}$ and $\bar v_{k3}:=v_{k3}$ for $k =1, \ldots, N_{1}$. This part 1 deals with the intersection points of the first $N_1$ grim reapers and ignores the rest of the configuration until the definition of $\overline{G}_1$. 

\begin{definition}
\label{def:level}
For an intersection point $p \in G$, we define its \emph{level} (see Figure \ref{fig1}):
\begin{itemize}
\item Level $1$: an intersection point is at \emph{level $1$} if it is in $\overline{\mathcal B}_1$. 
\item Level $j$: an intersection point is at \emph{level $j$} if it is not in any level below $j$ and if it is the endpoint of two different edges with endpoints at level $i_1, i_2 \leq j-1$. 
\item Level $\infty$: an intersection point is at  \emph{level $\infty$} if it is not at any finite level.
\end{itemize}
\end{definition}
The number of nonempty levels is finite because the number of intersection points is finite. Note also that an intersection point $\{p\} = \wilde \Gamma_n \cap \wilde \Gamma_{n'}$ has finite level if and only if $n, n' \in \{1, \ldots, N_1\}$. 

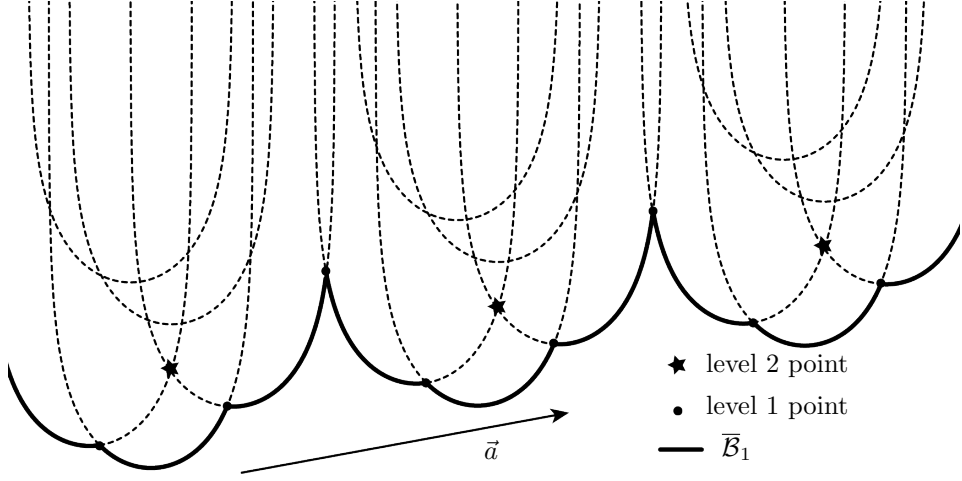
\begin{figure}
\includegraphics[width=\textwidth]{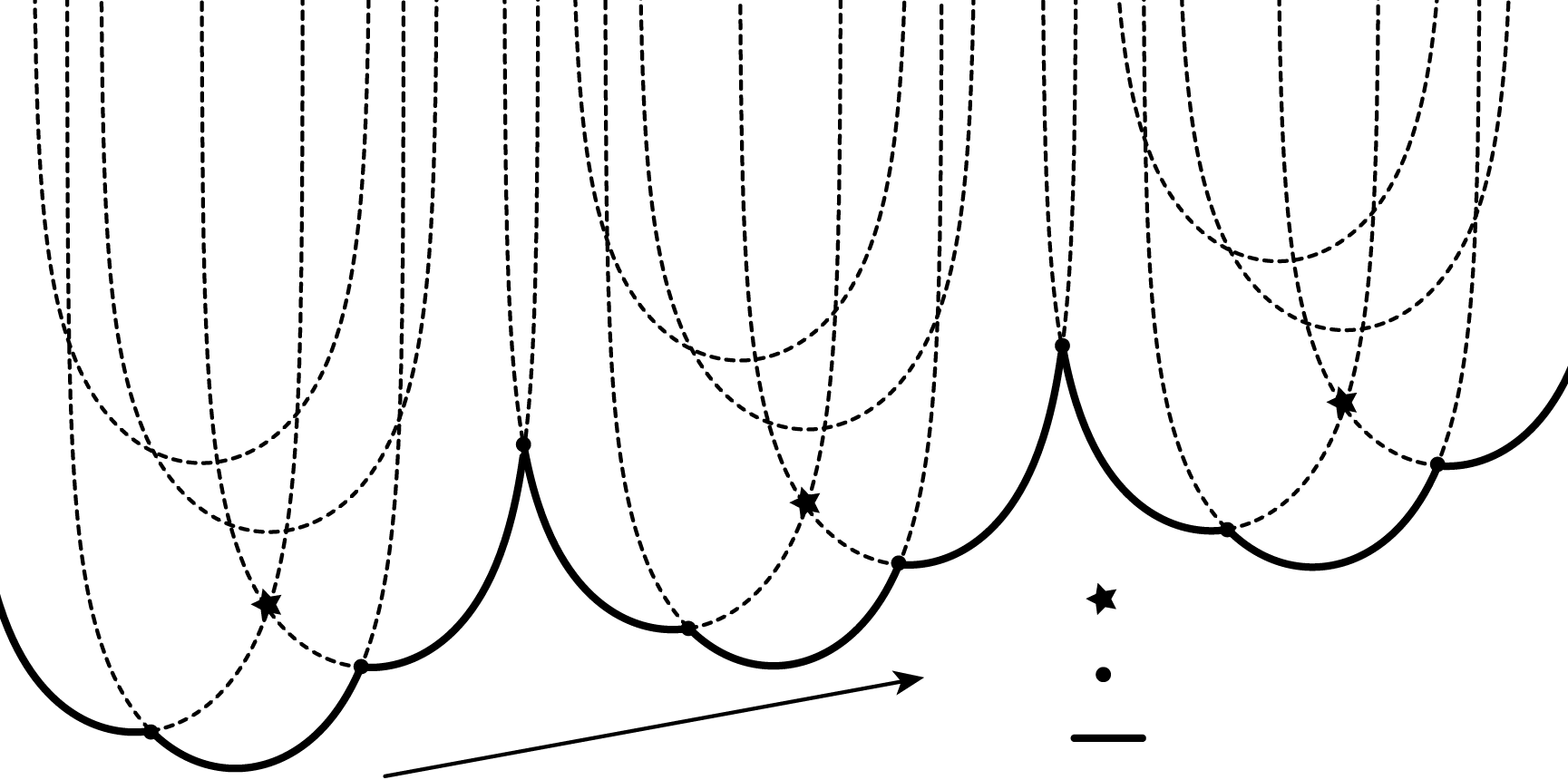}
\caption{Levels of the points of intersection in Part 1. The intersection points not marked are at level $\infty$.}
\label{fig1}
\begin{picture}(0,0)
\put(84,83){level 2 point }
\put(84,67){level 1 point}
\put(90, 51){$\overline{\mathcal B}_1$}
\put(0, 50){$\vec a$}
\end{picture}
\end{figure}

\emph {Algorithm for constructing a partial flexible configuration.}
We now show how to construct the new level $j$ intersection points assuming that the points $\bar p_k$ corresponding to intersections of level $j-1$ or lower are well-defined as well as their directing vectors $\bar v_{k2}$ and $\bar v_{k3}$. At these $\bar p_k$,  the directions $\bar v_{k1}$ and $\bar v_{k4}$ are determined by $\theta_{k,1}$ and $\theta_{k,2}$. We attach two pieces of grim reapers at $\bar p_k$, with tangent directions given respectively by $\bar v_{k1}$ and $\bar v_{k4}$. The intersection of the new pieces of grim reapers give us intersection points $\bar p_l$ corresponding to a level $j$ point $p_l \in G$. The second and third directing vectors at $\bar p_l$ are given by the newly modified grim reaper pieces. 

Using the algorithm, we construct new intersection points corresponding to points of $G$ with finite level. We complete Part 1 by finding the directions $\bar v_{k1}$ and $\bar v_{k4}$ for all the points with highest finite level and attaching modified pieces of grim reapers with tangents in these directions. We now have a flexible adjustment $G_1$ of $\bigcup_{j\in \mathbf Z}\bigcup_{n=1}^{N_1} \wilde \Gamma_1$ given the unbalancing at the intersection points of finite level. The configuration $G_1$ contains all the finite level modified intersection points and all the modified edges and rays emanating from these points. If there is no unbalancing, $G_1$ would just be $ \bigcup_{j\in \mathbf Z}\bigcup_{n=1}^{N_1} \wilde \Gamma_1$. We define $G_2^{\ast} :=\bigcup_{j\in \mathbf Z} \bigcup_{n=N_1+1}^{N_{\Gamma}} \wilde \Gamma_n + j \vec a$ to be the union of the other grim reapers and 
	\[
	\overline G_1 = \text{the connected component of } G_1 \setminus G_2^{\ast} \text{ containing }\overline{\mathcal B}_1.
	\]
Figure \ref{fig1G1} shows $\overline G_1$ in the case where there is no unbalancing. With unbalancing,  the edges, rays, and position of intersection points of level $i$, $1<i<\infty$, would be slightly perturbed. Note that if $N_1 < N_{\Gamma}$, $\overline{G}_1$ is not closed because it does not contain the points at level infinity.

\begin{figure}
\includegraphics[width=\textwidth]{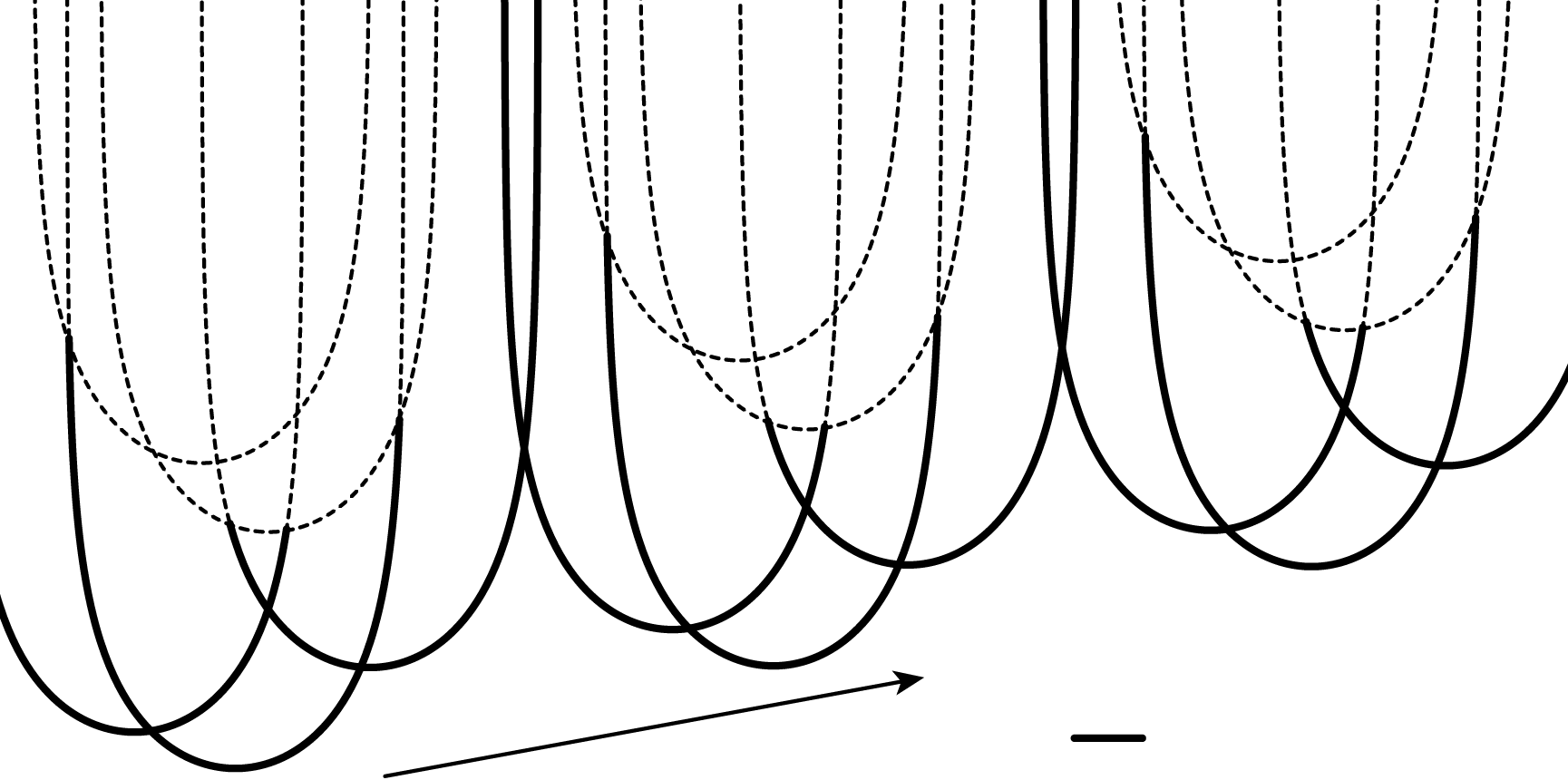}
\caption{The graph $\overline{G}_1$ when there is no unbalancing.}
\label{fig1G1}
\begin{picture}(0,0)
\put(90, 42){$\overline{G}_1$}
\put(0, 41){$\vec a$}
\end{picture}
\end{figure}
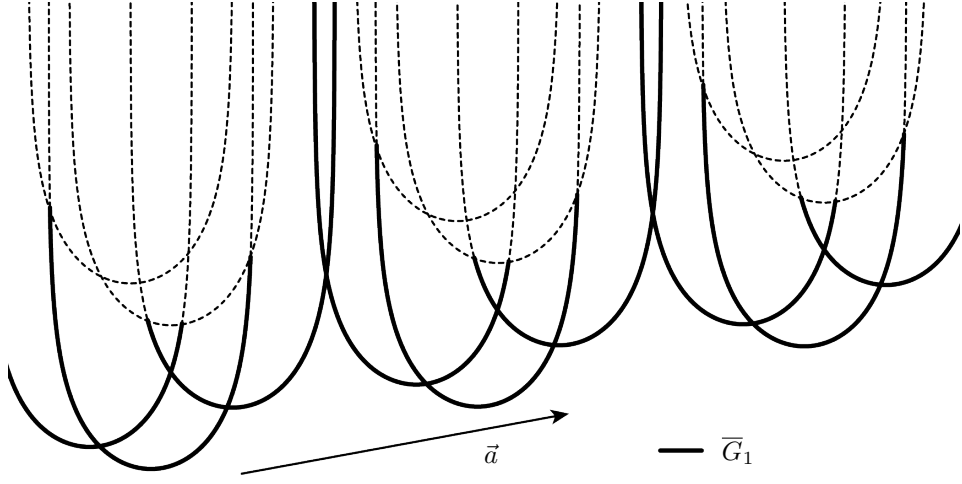

Let us assume that Part (m-1) has been completed and we obtained the partial configuration $\overline G_{i}$, $i=1, \ldots, m-1$ involving the grim reapers $\wilde \Gamma_n$, $n \leq N_{m-1}<N_{\Gamma}$. 

{\bf Part m.} We define  $G^{\ast}_m: = \bigcup_{j \in \mathbf Z} \bigcup_{n>N_{m-1}} \wilde \Gamma_n + j \bfa $ to be union of the remaining grim reapers and the bottom of $G^{\ast}_m$ by
	\[
	\mathcal B_m = \{ q \in G^{\ast}_m \mid y(q) = \min_{q' \in G, x(q')= x(q)} y(q')\}.
	\]
	
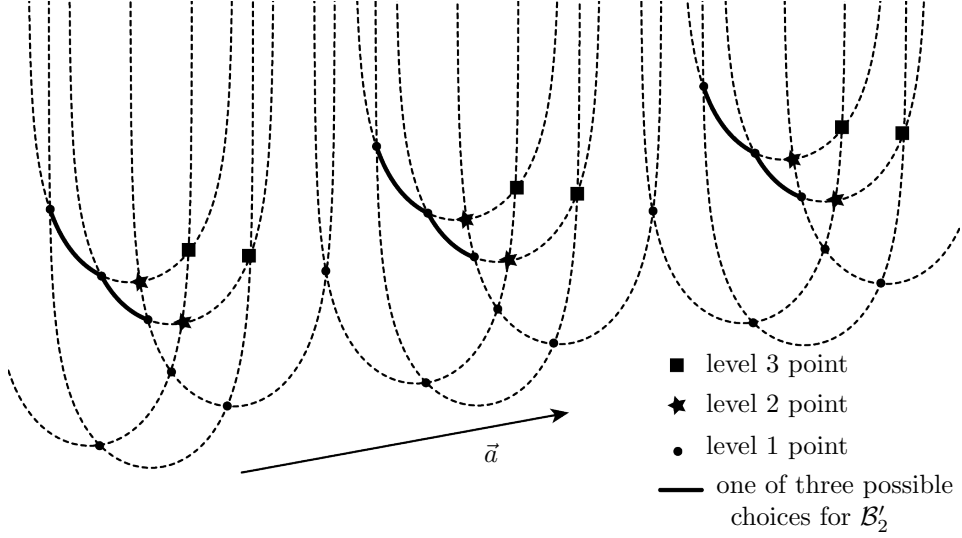
\begin{figure}
\includegraphics[width=\textwidth]{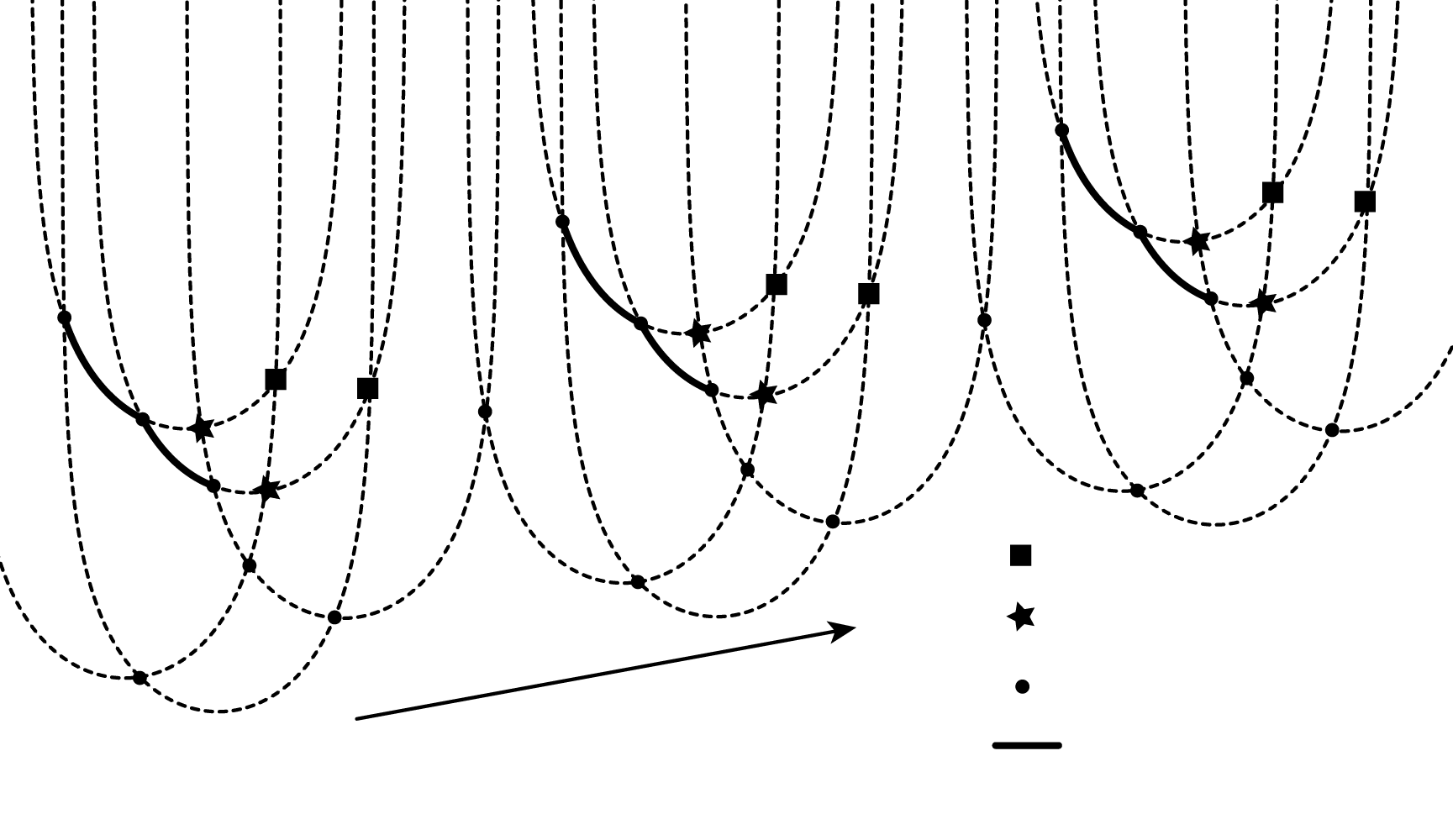}
\caption{Levels of the points of intersection in Part 2 for this choice of $\overline{\mathcal B}_2 = \mathcal B_2' \cup \overline{G}_1$. Not to overcrowd this figure, we assume there was no unbalancing in Part 1.}
\label{fig2}
\begin{picture}(0,0)
\put(84, 124){level 3 point}
\put(84,109){level 2 point }
\put(84,93){level 1 point}
\put(88, 78){one of three possible}
\put(90, 66)  { choices for ${\mathcal B}'_2$}
\put(0, 90){$\vec a$}
\end{picture}
\end{figure}

Let ${\mathcal B}'_m$ be the closure of an arbitrary bounded connected component of $\mathcal B_m \setminus (\text{the closure of } \overline G_{(m-1)})$ (see Figure \ref{fig2}). We define 
	\[
	\overline{\mathcal B}_m :=  \mathcal{B}'_m \cup \bigcup_{i=1}^{m-1} \overline G_{i}.
	\]
Any intersection point of  $\overline{\mathcal B}_m$ is either (i) in the interior of $\bigcup_{i=1}^{m-1} \overline G_{i}$, (ii) in the interior of $\mathcal{B}'_m$, or (iii) in $\mathcal{B}'_m \cap \text{closure}( \bigcup_{i=1}^{m-1} \overline G_{i})$. In case (i), we already have the correct unbalancing at the intersection; for points in case (ii) or (iii),  we fix the two directing vectors pointing to the interior of $\overline {\mathcal B}_m$. These vectors are not necessarily in the second and third directions but they are not opposite. We can assume by renumbering if necessary, that $\mathcal B_m$ contains edges from the grim reapers $\wilde \Gamma_n$, $N_{m-1} < n\leq N_m$. 

As in part 1, we want to redefine the level of each intersection point with respect to $\overline{\mathcal B}_m$. But $\overline{\mathcal B}_m$ is not a subset of $G$ so we unfortunately have to be more technical. We therefore define the level of each intersection point $p \in G$ by using Definition \ref{def:level} with $\overline{\mathcal B}_1$ replaced by the collection of edges and rays in $G$ that corresponds to edges or rays in $\overline{\mathcal B}_m$  (see Figure \ref{fig2}).

We can use the \emph {algorithm for constructing a partial flexible configuration} for constructing new level $j$ intersection points assuming that the intersection points  corresponding to level $(j-1)$ or lower are already defined, as well as two of their directing vectors. Although the two fixed directing vectors are not necessarily pointing in the second and third direction anymore, the algorithm is easily modified. We finish Part m by finding the last two directing vectors for all the points with highest finite level and attaching modified pieces of grim reapers with tangents in these directions. The resulting partial flexible configuration is denoted by $G_m$.  We define $G_{m+1}^{\ast} :=\bigcup_{j\in \mathbf Z} \bigcup_{n=N_m+1}^{N_{\Gamma}} \wilde \Gamma_n + j \vec a$ to be the union of the other grim reapers and 
	\[
	\overline G_m = \text{closure of the connected component of } G_m \setminus G_{m+1}^{\ast} \text{ containing }\overline{\mathcal B}_m.
	\]

There are finitely many parts because the number of intersection points in one period is finite and in each part, we modify at least one additional intersection point. The  existence of the constant $\delta'_{\theta}$ follows from Proposition 5 \cite{mine;self-trans}, which states the smooth dependence between the position of the intersection points, grim reaper positions, and directing vectors. 
\end{proof}


\section{Initial surfaces and the perturbation problem}
\label{sec:perturbation}

The rest of construction follows along the steps of \cite{mine;self-trans} is outlined them here.

We construct desingularizing surfaces to replace intersection lines by appropriately dislocating and bending original Scherk minimal surfaces. The dislocation is to ensure that the desingularizing surface can fit into an unbalanced configuration and to deal with small eigenvalues of the linear operator. The main purpose of the bending is to guarantee exponential decay of solutions to the linearized equation on this piece of surface.

We define a family of initial surfaces, which are approximate solutions to \eqref{eq:soliton}, by fitting appropriately bent and scaled Scherk surfaces at the intersection lines. The fitting of the desingularizing Scherk surfaces moves the edges and the grim reapers ends slightly; however, it does not change the position of the intersection lines so the periodicity can be preserved. The errors created can be controlled (see Proposition 33 and Corollary 34  \cite{mine;self-trans}).

We study the linear operator $\lin$ associated to normal perturbations of $H - \vec e_y \cdot \nu$. The linearized equation $\lin v =E$  can be solved, modulo the addition of correction functions to the inhomogeneous term. These correction functions can be generated by the flexibility build into $\overline G$ and the construction of the desingularizing surfaces (see Propositions 38 and 45 \cite{mine;self-trans}).

With the correct norms on the initial surfaces, we apply a fixed point theorem to find exact solutions (see Theorem 48 \cite{mine;self-trans}). The number of parameters introduced during the construction of the initial approximate solutions is finite because the number of intersections within a period is finite. At this point, it is worth noting that the obstacle that prevents us from constructing a infinite dimensional family of self-translating surfaces based on an initial configuration such as
	\[
	\bigcup_{j \in \mathbf Z} \wilde \Gamma + j \bfa + \bfb_j,
	\]
where $\bfb_j$ are vectors in $\R^2\times\{0\}$ with $|\bfb_j| < \delta$, is not the lack of flexibility but the fact that we would have infinitely many parameters. Even if the parameters are uniformly bounded, the current proof, which uses a fixed point theorem in the final step, would not apply here because a countably infinite product of compact intervals is not a compact Banach space.

\def\cprime{$'$}

\end{document}